\documentclass{amsart}
\usepackage{amsmath}
\usepackage{amssymb}
\usepackage[dvips]{graphicx,color}
\usepackage{amsthm}
\usepackage{amscd}
\usepackage{amsfonts, braket}
\usepackage[all]{xy}

\DeclareMathOperator{\Hom}{Hom}

\DeclareMathOperator{\Ima}{Im}
\DeclareMathOperator{\Ker}{Ker}
\DeclareMathOperator{\Aut}{Aut}

\DeclareMathOperator{\End}{End}

\DeclareMathOperator{\Der}{Der}

\DeclareMathOperator{\Prim}{Prim}

\DeclareMathOperator{\ad}{ad}
\DeclareMathOperator{\Ad}{Ad}

\DeclareMathOperator{\Diff}{Diff}

\DeclareMathOperator{\sk}{sk}

\DeclareMathOperator{\gr}{gr}

\DeclareMathOperator{\MC}{MC}

\DeclareMathOperator{\QAut}{QAut}
\def\z{\mathbb{Z}}

\def\r{\mathbb{R}}
\def\c{\mathbb{C}}

\def\Sp{\mathrm{Sp}}

\def\Sh{\mathrm{Sh}}

\def\A{\mathcal{A}}
\def\B{\mathcal{B}}

\def\D{\mathcal{D}}
\def\E{\mathcal{E}}
\def\F{\mathcal{F}}
\def\G{\mathcal{G}}
\def\H{\mathcal{H}}
\def\M{\mathcal{M}}

\def\Q{\mathcal{Q}}

\def\X{\mathcal{X}}

\def\QA{\mathcal{QA}}
\def\CD{\mathcal{CD}}
\def\m{\mathfrak{m}}

\def\o{\mathfrak{o}}

\def\id{\mathrm{id}}

\def\pt{\partial}
\def\Ab{\mathrm{Ab}}
\def\ct{\hat{L}W}

\def\sus{\mathbf{s}}

\theoremstyle{definition}
  \newtheorem{thm}{Theorem}[section]
  
  \newtheorem{lem}[thm]{Lemma}
  
  \newtheorem{prop}[thm]{Proposition}
\theoremstyle{definition}
  \newtheorem{defi}[thm]{Definition}

  \newtheorem{rem}[thm]{Remark}

\title{Obstruction of $C_\infty$-algebra models and characteristic classes}

\author{Takahiro Matsuyuki}
\address{Department of Mathematics, 
Tokyo Institute of Technology, 
2-12-1 Oh-okayama, Meguro-ku, Tokyo 152-8551, Japan.}
\email{matsuyuki.t.aa@m.titech.ac.jp}

\begin{document}
\maketitle
\begin{abstract}
In this paper, we consider an obstruction-theoretical construction of characteristic classes of fiber bundles by simplicial method. We can get a certain obstruction class for a deformation of $C_\infty$-algebra models of fibers and a characteristic map from the exterior algebra of a vector space of derivations. Applying this construction for a surface bundle, we obtain the Euler class of a sphere bundle and the Morita-Miller-Mumford classes of a bundle with positive genus fiber.
\end{abstract}

\section{Introduction}

Our purpose of the paper is to construct characteristic classes of a smooth fiber bundle $X\to E\to B$ by obstruction theory for a certain simplicial bundle $Q_\bullet(X)\to \Q_\bullet(E)\to S_\bullet(B)$ obtained from the original bundle as follows.
\begin{itemize}
\item The $n$-th base set $S_n(B)$ is the set of singular simplices $\Delta^n\to B$, 
\item The $n$-th fiber set $\Q_n(E)_\sigma$ over an $n$-simplex $\sigma\in S_n(B)$ is the set of Chen's formal homology connections \cite{C1,C2} on $\sigma^*E$. The total simplicial set $\Q_\bullet(E)$ is defined by the disjoint sum
\[\Q_n(E)=\coprod_{\sigma\in S_n(B)}\Q_n(E)_\sigma.\]
\item The $n$-th typical fiber $Q_n(X)$ is the set of Chen's formal homology connections $(\omega,\delta)$ on $X\times\Delta^n$. It has the decomposition with respect to differentials of formal homology connections
\[Q_\bullet(X)=\coprod_\delta Q_\bullet(X,\delta).\]
\end{itemize}
A formal homology connection is well-known as a main tool of the de Rham homotopy theory. It has homotopical information of $X$, which is equivalent to a minimal $C_\infty$-algebra model $f:(H,m) \to A$ (\cite{GLS}). Here $A$ is the de Rham complex of $X$, and $H$ is the de Rham cohomology of $X$. 

To apply the obstruction theory to the simplicial bundle, we need to calculate the homotopy group of $Q_\bullet(X)$. This simplicial set is very close to the Maurer-Cartan simplicial set $\MC_\bullet(\ct\otimes A)$ of the DGL $\ct\otimes A$, where $(\ct,\delta)$ is the dual of the bar-construction of the $C_\infty$-algebra $(H,m)$. The homotopy group of the Maurer-Cartan simplicial set is known in \cite{Get, B1, BFMT}. Using the results, we shall prove the homotopy groups of $Q_\bullet(X)$ are described as vector spaces by
\[\pi_n(Q_\bullet(X),\tau)\simeq H_n(\Der(\ct),\ad(\delta))=:\H_n(\delta)\]
for a formal homology connection $\tau=(\omega,\delta)$ on $X$ and $n\geq 1$. The set $\pi_0(Q_\bullet(X,\delta))$ of connected components can also be identified with a certain subspace $\H^{(1)}_0(\delta)$ of the $0$-th homology $\H_0(\delta)$ of the DGL $(\Der(\ct),\ad(\delta))$. These calculations of the homotopy groups of $Q_\bullet(X)$ are shown in Section \ref{FHC}.

If $Q_\bullet(X)$ is $(n-1)$-connected, under certain conditions, we can construct the obstruction class of existence of a partial section over the $n$-skeleton of $\Q_\bullet(E)\to S_\bullet(B)$ 
\[\o_n\in H^{n+1}(B;\Pi_n),\]
where $\Pi_n$ is the local system of the $n$-th homotopy groups of fibers of $\Q_\bullet(E)\to S_\bullet(B)$. Then, by contracting coefficients of $\o_n$, we can also the characteristic map
\[(\Lambda^p\H_n(\delta)^*)^G\to H^{p(n+1)}(B;\r)\]
for any $p\geq1$. Here $G$ is the structure group of $E\to B$. For example, the image of characteristic map for the sphere bundle associated to the Hopf fibration is generated by the Euler class of this bundle. Their definitions are written in Section 5.1, and the example is in Section 5.2.

On the other hand, if $Q_\bullet(X)$ is not connected, instead of $\o_n$, there exists $i\geq 1$ and the cohomology class $\o^{(i)}\in H^1(B;\gr_i(\QA^+(E)))$. Here $\QA^+(E)$ is a certain local system of groups with a filtration. The local system $\Pi_0$ of the $0$-th homotopy set of fibers of $\Q_\bullet(E)\to S_\bullet(B)$ has a structure of $\QA^+(E)$-torsor, i.e., a free and transitive action of $\QA^+(E)$. So we can also construct  the characteristic map
\[(\Lambda^\bullet\gr_i(\H_0(\delta))^*)^G\to H^\bullet(B;\r).\]
Applying for a surface bundle, we get the obstruction class $\o^{(1)}$, and it is equal to the 1st twisted Morita-Miller-Mumford class. It means that the characteristic map gives Morita-Miller-Mumford classes. Their definitions are written in Section 5.3, and the application for a surface bundle is in Section 5.4.


The paper is organized as follows. 
\begin{itemize}
\item In Section 2, we define terms used in the paper. 
\item In Section 3, we define the simplicial set of formal homology connections, and calculate its homotopy groups.
\item In Section 4, we describe obstruction theory for general simplicial sets in order to use in Section 5. This section is independent of other sections.
\item In Section 5, we apply the discussions in Section 4 for $\Q_\bullet(E)\to S_\bullet(B)$ and get obstruction classes. We also calculate the obstruction classes for specific bundles.
\end{itemize}

{\bf Acknowledgment.} I would like to thank my supervisor Y. Terashima for many helpful comments. This work was supported by Grant-in-Aid for JSPS Research Fellow (No.17J01757).

\section{Preliminary}

In this paper, all vector spaces are over the real number field $\r$. The standard $n$-simplex is described by
\[\Delta^n=\left\{(t_i)_{i=0}^n\in \r^{n+1};\sum_{i=0}^nt_i=1\right\}.\]
Fix its base point $\{(1,0,\dots,0)\}=\delta^n\cdots \delta^1(\Delta^n)$, where $\delta^i:\Delta^{n-1}\to \Delta^n$ is the $i$-th coface operator.

Throughout the paper, we consider a smooth fiber bundle $X\to E\to B$ whose fiber $X$ is a manifold with a base point, i.e., a smooth fiber bundle with a smooth section $B\to E$. We always suppose that 
\begin{itemize}
\item a manifold $X$ is connected,
\item its base point $*$ is fixed, and
\item its rational homology group is finite-dimensional. 
\end{itemize}

The structure group of the bundle, which is a subgroup of the diffeomorphism group $\Diff(X)$, acts on the homology group of $X$. We call its image $G$ in the automorphism group of $H_\bullet(X;\r)$ the \textbf{homological structure group}.

\subsection{Graded vector space}
Let $V$ be a $\z$-graded vector space. We denote $V^i$ the subspace of elements of $V$ of \textbf{cohomological degree} $i$ and $V_i=V^{-i}$ the subspace of elements of \textbf{homological degree} $i$. Remark that the \textbf{linear dual} $V^*=\Hom(V,\r)$ of $V$ is graded by $(V^*)^i=\Hom(V_i,\r)$. For an element $v\in V$, its homological degree is denoted by $|v|$.

The \textbf{$p$-fold suspension} $V[p]$ of $V$ for an integer $p$ is defined by 
\[V[p]^i:=V^{i+p}.\]
In the case of $p=1$, elements of $V[1]^i$ are written by $\sus x$ for $x\in V^{i+1}$ using the symbol $\sus$ of cohomological degree $-1$. The symbol $\sus$ is also regarded as the map $\sus:V\to V[1]$ with cohomological degree $-1$. We often use the inverse map $\sus^{-1}:V[1]\to V$ of the suspension $\sus$ too.

\subsection{FDGL}\label{FDGL}
In Section \ref{FLA} and \ref{der}, we use the following structures:
\begin{defi}
Let $L$ be a $\z$-graded Lie algebra and $\F=\{\F^{(i)}\}_{i=0}^\infty$ a decreasing filtration of $L$. The pair $(L,\F)$ is called a \textbf{$\z$-graded filtered Lie algebra}, \textbf{FGL} for short, if it satisfies $[\F^{(i)},\F^{(j)}]\subset \F^{(i+j)}$ for integers $i,j\geq 0$. Moreover, given the differential $\delta$ on $L$ with homological degree $-1$, the triple $(L,\delta,\F)$ is called a \textbf{filtered DGL}, \textbf{FDGL} for short, if it satisfies $\delta(\F^{(i)})\subset \F^{(i+1)}$ for integers $i\geq 0$. Then the homology $H_\bullet(L,\delta)$ has the canonical FGL structure whose filtration $\bar{\F}=\{\bar{\F}^{(i)}\}_{i=0}^\infty$ is defined by 
\[\bar{\F}^{(i)}:=\Ima(\Ker(\delta)\cap \F^{(i)}\to H_\bullet(L,\delta)).\]
\end{defi}
Then the Lie algebra of derivations on a FGL (resp. FDGL) is a FGL (resp. FDGL) as follows:
\begin{defi}\label{homologyFDGL}
Let $(L,\F)$ is a FGL. Put
\[\Der(L)_n:=\{D\in \End(L);D([x,y])=[D(x),y]+(-1)^{|x|n}[x,D(y)],\ D(L_p)\subset L_{p+n}\},\]
\[\Der(L):=\bigoplus_n\Der(L)_n,\]
\[\D=\{\D^{(i)}\}_{i\geq0},\quad\D^{(i)}:=\{D\in \Der(L);D(\F^{(q)})\subset \F^{(q+i)}\}.\]
Then the space $\Der(L)$ of derivations on $L$ is a $\z$-graded Lie subalgebra of $\End(L)$, and $(\Der(L),\D)$ is a FGL. If $(L,\delta,\F)$ is a FDGL, then so is $(\Der(L),\ad(\delta),\D)$.
\end{defi}

\subsection{Free Lie algebra}\label{FLA}
Let $W$ be a $\z$-graded vector space. In this paper, $W$ is always homologically non-negatively graded. 

The graded free Lie algebra $LW$ and the completed free Lie algebra $\ct$ generated by $W$ have the canonical FGL structures as follows:
\begin{itemize}
\item (grading) These two Lie algebras $LW$ and $\hat{L}W$ can be defined as the primitive part of the tensor algebra $TW$ and the completed tensor algebra $\hat{TW}$:
\[LW=\Prim TW=\{x\in TW;\Delta(x)=1\otimes \Delta(x)+\Delta(x)\otimes 1\},\]
\[\ct=\Prim \hat{T}W=\{x\in \hat{T}W;\Delta(x)=1\otimes \Delta(x)+\Delta(x)\otimes 1\},\]
where $\Delta$ is the (completed) coproduct. Since the algebra $TW$ is $\z$-graded by
\[(TW)_n:=\bigoplus_{p\geq 0}\bigoplus_{i_1+\cdots+i_p=n}(W_{i_1}\otimes\cdots \otimes W_{i_p}),\]
so is the Lie algebra $TW$:
\[(LW)_n:=LW\cap (TW)_n.\]
The Lie algebra $\ct\subset \hat{T}W$ is $\z$-graded in the same way.
\item (filtration) Let $\Gamma=\{\Gamma_n\}_{n=1}^\infty$ be the lower central series of $LW$ and $\hat{\Gamma}=\{\hat{\Gamma}_n\}_{n=1}^\infty$ the completed lower central series of $\ct$. It is described by
\[\Gamma_n=LW\cap \bigoplus_{m>n}  W^{\otimes m},\quad\hat{\Gamma}_n=\ct\cap \prod_{m>n}  W^{\otimes m}.\]
Note that
\[\hat{\Gamma}_n=\varprojlim_k\Gamma_n/\Gamma_{n+k},\quad\hat{L}W=\varprojlim_kLW/\Gamma_{k+1}.\]
\end{itemize}
Then $(LW,\Gamma)$ and $(\ct,\hat{\Gamma})$ are FGLs.

\subsection{Derivations}\label{der}
Fix a differential $\delta$ on $\ct$ satisfying $|\delta|=-1$ and $\delta(W)\subset \hat{\Gamma}_2$. Then, from Section \ref{FDGL}, the triples $(\ct,\delta,\hat{\Gamma})$ and $(\Der(\ct),\ad(\delta),\D)$ are FDGLs. Its homology $\H_\bullet(\delta):=H_\bullet(\Der(\ct),\ad(\delta))$ has the induced filtration $\bar{\D}$ as in Definition \ref{homologyFDGL}. Especially, we have the Lie algebra $\H_0(\delta)$ filtered by $\H_0^{(i)}(\delta):=\bar{\D}_0^{(i)}$.

\begin{defi}[exponential map]Consider the group of automorphisms $\Aut(\ct)$ of the completed Lie algebra $\ct$ filtered by
\[\A^{(i)}:=\Ker(\Aut(\ct)\to \Aut(\ct/\hat{\Gamma}_{i+1})).\]
Then the bijection $\exp :\D^{(1)}_0\to \A^{(1)}$ preserving their filtrations, which is called the \textbf{exponential map}, defined by
\[\exp (D)=\sum_{n=0}^\infty \frac{D^n}{n!}\in \End(\ct).\]
The map has the inverse map $\log:\A^{(1)}\to\D_0^{(1)}$. The product of the group $\A^{(1)}$ and the Lie bracket of $\D^{(1)}_0$ are related by the Baker-Campbell-Hausdorff formula. We can also consider the group of automorphisms $\Aut(\delta)$ of the completed dgl $(\ct,\delta)$ filtered by $\A^{(i)}(\delta):=\A^{(i)}\cap \Aut(\delta)$ and the Lie algebra
\[\Der(\delta):=\Ker(\ad(\delta):\Der(\ct)\to \Der(\ct)),\]
which is filtered by $\D^{(i)}(\delta):=\D^{(i)}\cap \Der(\delta)$. Then we can get the restriction $\exp:\D^{(1)}_0(\delta)\to \A^{(1)}(\delta)$ of the exponential map $\D^{(1)}_0\to \A^{(1)}$. So we can get the quotient group
\[\QAut(\delta):=\Aut(\delta)/\exp(\ad(\delta)(\Der(\ct)_1)),\]
which is filtered by the image $\QA(\delta)$ of the filtration $\{\A^{(i)}(\delta)\}_{i\geq 0}$, and the induced exponential map $\exp:\H_0^{(1)}(\delta)\to \QA^{(1)}(\delta)$.
 \end{defi}

\begin{defi}
Let $G$ be a group with a decreasing filtration $\{G^{(i)}\}_{i\geq0}$ of normal subgroups satisfying 
\[[G^{(i)},G^{(j)}]\subset G^{(i+j)},\quad G^{(0)}=G.\]
Then $\gr_i(G):=G^{(i)}/G^{(i+1)}$ is an abelian group with respect to the sum induced by the product of $G$, and  
\[\gr(G):=\bigoplus_{i=0}^\infty \gr_i(G)\]
is a Lie algebra with respect to the Lie bracket defined by the commutator of $G$. Similarly, for a filtered Lie algebra $(L,\F)$, we get the Lie algebra 
\[\gr(L):=\bigoplus_{i=0}^\infty \gr_i(L),\quad \gr_i(L):=\F^{(i)}/\F^{(i+1)}.\]

\end{defi}

Using the notations above, the exponential map induces the isomorphism 
\[\gr_i(\H_0(\delta))\simeq \gr_i(\QA(\delta))\]
for $i\geq 1$.

If $\delta(W)\subset [W,W]$, we can define another grading of $\Der(\ct)$ by
\[\Der^i(\ct):=\{D\in \Der(\ct);\ D(W)\subset W^{\otimes(i+1)}\}.\] 
Then $\ad(\delta)$ has the degree $1$ with respect to the grading. So we have the canonical identification 
\[\H_0^i(\delta):=H^i_0(\Der(\ct),\ad(\delta))\simeq \gr_i(\H_0(\delta)),\]
where $i$ is the second grading of $\Der(\ct)$.

\subsection{Formal homology connections}

In this subsection, we shall review the definition of a formal homology connection on a manifold $X$. We denote the de Rham complex on $X$ by $A^\bullet(X)$, the reduced de Rham complex and cohomology by
\[A=\tilde{A}^\bullet(X):=\Ker(A^\bullet(X)\to A^\bullet(*)=\r),\quad H=\tilde{H}^\bullet_{DR}(X)\]
and the suspension of the reduced real homology by $W=\tilde{H}_\bullet(X;\r)[-1]$.

\begin{defi}[Chen \cite{C1, C2}]
A \textbf{formal homology connection} on $X$ is a pair $(\omega,\delta)$ satisfying the following conditions:
\begin{enumerate}
\item an $\ct$-coefficient differential form $\omega\in A\otimes \ct$ with cohomological degree $1$ is described by 
\[\omega=\sum_{k=1}^{\infty}\sum_{i_1,\dots,i_k}\omega_{i_1\cdots i_k}x_{i_1}\cdots x_{i_k},\]
where $x_1,\dots,x_n$ is a homogeneous basis of $W$, such that 
\[\int_{x_p}\omega_p=1.\]
\item  a linear map $\delta:\ct\to \ct$ is a differential with homological degree $-1$ of $\ct$ such that 
\[\delta(W)\subset \hat{\Gamma}_2.\]
\item the form $\omega$ is a Maurer-Cartan element of $(A\otimes \ct,d+\delta)$, i.e., the flatness condition $\delta\omega+d\omega+\frac12 [\omega, \omega]=0$ holds. (Though the sign notation may be different from Chen's original definition, they are equivalent.)
\end{enumerate}
We call such a differential $\delta$ \textbf{Chen differential} of $X$. If $X$ is simply connected, we can replace the free Lie algebra $LW$ and its derivation $\delta:LW\to LW$ with $\ct$ and $\delta:\ct\to\ct$ respectively. 
\end{defi} 

According to Chen \cite{C1, C2}, a Riemannian metric on $X$ defines the canonical formal omology connection on $X$. If $X$ is a formal manifold (e.g. K\"ahler manifold), there exists a Chen differential $\delta$ satisfying $\delta(W)\subset [W,W]$. This differential is corresponding to the usual product of the cohomology $H$ by method of Section \ref{homotopyalg}.

We can pull-back a formal homology connection by a diffeomorphism as follows. So the diffeomorphism group of $X$ acts on the set of formal homology connections on $X$.
\begin{defi}[pull-back of formal homology connections]
Let $\varphi:Y\to X$ be a diffeomorphism preserving base points, and $(\omega,\delta)$ be a formal homology connection on $X$. Then we can define the formal homology connection on $Y$
\[\varphi^*(\omega,\delta):=((\varphi^*\otimes |\varphi|)\omega,|\varphi|\circ \delta\circ |\varphi|^{-1}),\]
where $|\varphi|:\hat{L}W(X)\to\hat{L}W(Y)$ is the induced map by $\varphi$, 
\[W(X)=\tilde{H}_\bullet(X;\r)[-1],\quad W(Y):=\tilde{H}_\bullet(Y;\r)[-1].\]

\end{defi}

\subsection{$C_\infty$-algebras and formal homology connections} \label{homotopyalg}
In the subsection, we shall review the definition of a $C_\infty$-algebra, and mention the relation between a formal homology connection and a $C_\infty$-algebra (\cite{GJ}). The definitions and discussions in the subsection are used in the proofs of Theorem \ref{homotopyQ} and \ref{ThKMT}. 
\begin{defi}[$C_\infty$-algebra]
Let $A$ be a vector space and $m=\{m_i\}_{i=1}^\infty$ be a family of linear maps $m_i:A^{\otimes i}\to A$ with degree $2-i$. The pair $(A,m)$ satisfying the following conditions is called a \textbf{$C_\infty$-algebra}:
\begin{itemize}
\item \textbf{($A_\infty$-relation)}
\[\sum_{k+l=i+1}\sum_{j=0}^{k-1}(-1)^{(j+1)(l+1)}m_k\circ (\id_A^{\otimes j}\otimes m_l\otimes \id_A^{\otimes (i-j-l)})=0\]
for $i\geq1$, and
\item \textbf{(commutativity)}
\[\sum_{\sigma\in \Sh(j,i-j)}\epsilon\cdot m_i(a_{\sigma(1)},\cdots,a_{\sigma(i)})=0\]
for $i>j>0$ and homogeneous elements $a_1,\dots,a_i\in A$, where $\Sh(i,i-j)$ is the set of $(i,i-j)$-shuffles and $\epsilon$ is the Koszul sign.
\end{itemize}
If $m_1=0$, $(A,m)$ is called \textbf{minimal}. If higher products are all zero, i.e. $m_3=m_4=\cdots =0$, $(A,m)$ can be regarded as differential graded commutative algebra (DGcA). 
\end{defi}

\begin{rem}[Bar construction of a $C_\infty$-algebra]
Let $(A,m)$ be a $C_\infty$-algebra and $\sus:A\to A[1]$ be the suspension map. We denote the tensor coalgebra $T^c(A[1])$ generated by $A[1]$ by $BA$. It is a bialgebra by the tensor coproduct $\Delta$ and the shuffle product $\mu$. Defining the suspension of $m_i$ by $\bar{m}_i:=\sus\circ m_i\circ (\sus^{-1})^{\otimes i}$ for all $i\geq 1$, then $\bar{m}_i:A[1]^{\otimes n}\to A[1]$ is degree $1$ and satisfies the commutativity condition. Thus extending the unique coderivation $\mathfrak{m}_i:BA\to BA$ by the co-Leibniz rule $\Delta\circ \m_i=(\m_i\otimes \id+\id\otimes\m_i)\circ\Delta$, then we have the Hopf derivation
\[\m:=\sum_{i=1}^\infty\m_i.\]
Furthermore $\m$ is a degree $1$ codifferential, i.e. $\m^2=0$, from the $A_\infty$-relations of $m$.

\end{rem}

\begin{defi}[$C_\infty$-morphism]
Let $(A,m)$ and $(A',m')$ be two $C_\infty$-algebras and $f=\{f_i\}_{i=1}^\infty$ be a family of linear maps $f_i:A^{\otimes i}\to A'$ with degree $1-i$ satisfying the following conditions:
\begin{itemize}
\item \textbf{($A_\infty$-morphism)}
\[\sum_{\substack{l\geq 1,\\k_1+\cdots+k_l=k}}(-1)^{\sum_{j=1}^lk_j(l-j)+\sum_{\nu<\mu}k_\nu k_\mu}m_l'\circ (f_{k_1}\otimes \cdots\otimes f_{k_l})\]\[=
\sum_{\substack{s+1+t=i,\\ s+l+t=k}}(-1)^{1+k+(s+1)(l+1)}f_i\circ (\id_A^{\otimes s}\otimes m_l\otimes \id_A^{\otimes t})\]
for $k\geq1$, and
\item \textbf{(commutativity)}
\[\sum_{\sigma\in \Sh(j,i-j)}\epsilon\cdot f_i(a_{\sigma(1)},\cdots,a_{\sigma(i)})=0\]
for $i>j>0$ and homogeneous elements $a_1,\dots,a_i\in A$.
\end{itemize}
Then $f$ is called a \textbf{$C_\infty$-morphism}. If $f_1$ is a quasi-isomorphism, $f$ is called a \textbf{$C_\infty$-quasi-morphism}. 
\end{defi}
\begin{defi}
Given $C_\infty$-algebra $(A,m^A)$, a pair $f:(H,m)\to (A,m^A)$ of a $C_\infty$-algebra structure $m$ on the cohomology $H:=H(A,m_1^A)$ and a $C_\infty$-quasi-isomorphism $f$ such that $f_1$ induces the identity map on the cohomology $H$ is called \textbf{$C_\infty$-algebra model} of $A$.
\end{defi}
\begin{rem}[Bar construction of a $C_\infty$-morphism]
Let $f:(A,m)\to (A',m')$ be a $C_\infty$-morphism. Defining the suspension of $f_i$ by $\bar{f}_i:=\sus\circ f_i\circ (\sus^{-1})^{\otimes i}:A[1]^{\otimes i}\to A'[1]$ for all $i\geq 1$, then the degree of $\bar{f}_i$ is $0$. Constructing the coalgebra map $BA\to BA'$
\[\mathfrak{f}:=\sum_{k=1}^\infty\sum_{\substack{i\geq 1,\\k_1+\cdots+k_i=k}}\bar{f}_{k_1}\otimes \cdots\otimes \bar{f}_{k_i}\]
from maps $\bar{f}_n$, we have the equations
\[\mathfrak{f}\circ\m=\m'\circ\mathfrak{f},\quad \mathfrak{f}\circ\mu=\mu\circ(\mathfrak{f}\otimes \mathfrak{f}).\]
So $\mathfrak{f}$ is a differential bialgebra map $(BA,\m)\to (BA',\m')$ between bar constructions. 
\end{rem}

According to \cite{GLS}, a formal homology connection $(\omega,\delta)$ on $X$ is equivalent to a minimal $C_\infty$-algebra model $f:(H,m)\to A$ of $A$. It is verified as follows: put
\[\omega=-\sum_{i_1,\dots,i_k}(-1)^\epsilon \sus^{-1}\bar{f}_n(x^{i_1},\dots,x^{i_k})x_{i_1}\cdots x_{i_k},\]
\[\delta=\m^*,\]
where 
\[\epsilon=|x_{i_1}|(|x_{i_2}|+\cdots+|x_{i_k}|)+\cdots+|x_{i_{k-1}}||x_{i_k}|,\] 
$\bar{f}_n=\sus f_n(\sus^{-1})^{\otimes n}:H[1]^{\otimes n}\to A[1]$, $x^i$ is the dual basis of $x_i$, and $\m$ is the bar-construction of $m$. Then the differential $\delta$ on the dual $(BH)^*=\hat{T}W$ of the bar-construction $BH$ can be restricted on $\ct$ since $\delta$ is a coderivation. So the pair $(\omega,\delta)$ is a formal homology connection on $X$. Conversely we can recover $f:(H,m)\to A$ from $(\omega,\delta)$. Note that the condition that $f$ is an $A_\infty$-morphism corresponds to the flatness. 

For a diffeomorphism $\varphi:Y\to X$, the formal homology connection $\varphi^*(\omega,\delta)$ on $Y$ is corresponding to the $C_\infty$-algebra model \[\varphi^*\circ f\circ (\varphi^*)^{-1}:(H(Y),\varphi^*\circ m\circ (\varphi^*)^{-1})\to A(Y).\] Here $H(Y)$ is the reduced de Rham cohomology of $Y$, and $A(Y)$ is the reduce de Rham complex of $Y$.

\section{The set of formal homology connections}\label{FHC}

In this section, we shall introduce the main tool of this paper, the simplicial bundle $Q_\bullet(X)\to \Q_\bullet(E)\to S_\bullet(B)$ associated to a fiber bundle $X\to E\to B$. 
\begin{itemize}
\item In subsection \ref{QX}, we shall define the Kan complex $Q_\bullet(X)$ and calculate its homotopy group.
\item In subsection \ref{QE}, we shall define the bundle $\Q_\bullet(E)\to S_\bullet(B)$ and give how to fix differentials of fibers.
\end{itemize}
The notions and the theorems in this section is used in Section \ref{mainsection}.

\subsection{The simplicial set of formal homology connections}\label{QX}

Let $X$ be a manifold. The set of formal homology connections on $X$ is denoted by $Q_0(X)$. 

We define the simplicial de Rham DGA $A_\bullet=\{A_n\}_{n=0}^\infty$ on $X$ by
\[A_n:=\tilde{A}^\bullet(X\times \Delta^n).\]
Its face maps and degeneracy maps are induced by the coface maps and the codegeneracy maps of the cosimplicial space $\Delta^\bullet=\{\Delta^n\}_{n=0}^\infty$.

Put $Q_n(X):=Q_0(X\times \Delta^n)$. The face map $Q_n(X)\to Q_{n-1}(X)$ and the degeneracy map $Q_{n}(X)\to Q_{n+1}(X)$ are induced by those of the simplicial DGA $A_\bullet$. Thus the family $Q_\bullet(X)=\{Q_n(X)\}_{n=0}^\infty$ of sets is a simplicial set. Given a Chen differential $\delta$ on $X$, the set of formal homology connections $(\omega,\delta)$ on $X\times \Delta^n$ is denoted by $Q_n(X,\delta)$. Then $Q_\bullet(X,\delta)$ is also a simplicial set. 

We denote the set of Maurer-Cartan elements of $(A_n\otimes\ct,d+\delta)$ by $\MC_n(X,\delta)$. We obtain the simplicial set $\MC_\bullet(X,\delta)$, and then $Q_\bullet(X,\delta)$ is a simplicial subset of $\MC_\bullet(X,\delta)$.

\begin{lem}\label{aboutQ}
For any $n$-th simplicial Maurer-Cartan element $\alpha\in \MC_n(X,\delta)$, if $\pt_i\alpha\in Q_{n-1}(X)$ for some $0\leq i\leq n$, then $\alpha\in Q_n(X,\delta)$.

\end{lem}
\begin{proof}
We can identify $\alpha$ with a $C_\infty$-map $f:H\to A_n$. The face map $\pt_i:A_n\to A_{n-1}$ for any $i$ gives the standard identification by $H^\bullet(X\times \Delta^n)\simeq H^\bullet(X\times \Delta^{n-1})$ and $H(\pt_if_1):H\to H(A_{n-1})$ is the identity map under the assumption. Therefore we have the commutative diagram 
\[\xymatrix{H\ar[r]^-{H(f_1)}\ar[dr]_{\id=H(\pt_if_1)}&H(A_n)\ar[d]^{H(\pt_i)=\id}\ar@{=}[r]&H\ar@{=}[d]\\&H(A_{n-1})\ar@{=}[r]&H}\]
and it leads $H(f_1)=\id:H\to H(A_{n})=H$.
\end{proof}

Since the simplicial set $\MC_\bullet(X,\delta)$ is a Kan complex (proved in Section 4 of \cite{Get}), the following lemma is obtained immediately from Lemma \ref{aboutQ}:
\begin{lem}\label{hoeq}
The simplicial set $Q_\bullet(X)$ is a Kan complex. Furthermore the map  induced by the inclusion\[\pi_0(Q_\bullet(X,\delta))\to \pi_0(\MC_\bullet(X,\delta))\] is injective, and the map \[\pi_n(Q_\bullet(X,\delta),\tau)\to \pi_n(\MC_\bullet(X,\delta),\tau)\] for $\tau\in Q_0(X,\delta)$ and $n\geq1$ is an isomorphism.

\end{lem}

\begin{thm}\label{homotopyQ}
The homotopy groups of the simplicial set $Q_\bullet(X)$ are described by
\[\pi_n(Q_\bullet(X),\tau)\simeq \H_n(\delta)\]
for $n\geq 1$ and a formal homology connection $\tau=(\omega,\delta)$ on $X$, where $\H_1(\delta)$ is equipped with the Baker-Campbell-Hausdorff product of $H_0(A\otimes \ct)$.

\end{thm}

\begin{proof}
From Proposition 5.4 and Theorem 5.5 in \cite{B1}, we have
\[\pi_n(Q_\bullet(X),\tau)\simeq\pi_n(\MC_\bullet(X,\delta),\tau)\simeq H_{n-1}(A\otimes \ct,d+\delta+[\omega,-]).\]
Let $F:BH\to BA$ be the bar-construction $F:BH\to BA$ of the $C_\infty$-morphism $H\to A$ corresponding to $\tau$. An endomorphism $D:BH\to BA$ is called a Hopf derivation over $F$ if it satisfies
\[D\nabla=\nabla(D\otimes F+F\otimes D),\quad \Delta D=(D\otimes F+F\otimes D)\Delta,\]
where $\nabla$ is the product and $\Delta$ is the coproduct. The vector space of Hopf derivations over $F$ is denoted by $\Der_F(BH,BA)$. A Hopf derivation $D$ over $F$ is uniquely determined by the lowest term $BH\overset{D}\to BA\overset{\text{proj.}}\to A[1]$. So there exists the natural inclusion $\Der_F(BH,BA)\subset \Hom(BH,A[1])$.

We shall prove the suspension of $(A\otimes \ct,d+\delta+[\omega,-])$ and the chain complex $(\Der_F(BH,BA),\mathfrak{D})$ are isomorphic. Here the differential $\mathfrak{D}$ is defined by
\[\mathfrak{D}(D)=\m^A\circ D-(-1)^DD\circ\m,\]
where $\m^A$ and $\m$ are the bar-constructions of the $C_\infty$-algebra structures $m^A$ and $m$ on $A$ and $H$ respectively. Here $m$ is the corresponding $C_\infty$-structure on $H$ to $\delta$.

Through the embedding $\ct\subset \hat{T}W=(BH)^*$, consider the linear isomorphism $\Phi:A[1]\otimes \ct\to \Der_F(BH,BA)\subset \Hom(BH,A[1])$ defined by
\[\Phi(\alpha\otimes f)(x)= f(x)\alpha\]
for $x\in BH$. Here the differential on $A[1]\otimes\ct$ is equal to $\sus(d+\delta+[\omega,-])\sus^{-1}$. Then, using $F=\Phi(\sus\omega)$, we have
\begin{align*}
&\Phi(\sus(d+\delta+[\omega,-])\sus^{-1}(\alpha\otimes f))(x)\\
=&d\alpha f(x)+(-1)^{\alpha+1}\alpha \delta f(x)+\sus[\omega,\sus^{-1}\alpha \otimes f](x)\\
=&d\alpha f(x)-(-1)^{\alpha+f}\alpha f\m(x)+\m_2^A\circ (F\otimes\Phi(\alpha  f))(x)+\m_2^A\circ(\Phi(\alpha f)\otimes F)(x)\\
=&\mathfrak{D}\Phi(\alpha f)(x).
\end{align*}
Thus the map $\Phi$ is a chain isomorphism. 

On the other hand, the map
\[F\circ -:(\Der(\ct),\ad(\delta))=(\Der(BH),\ad(\m))\to (\Der_F(BH,BA),\mathfrak{D})\]
is a quasi-isomorphism because $F$ is a quasi-isomorphism. So we get the isomorphism
\[ H_{n-1}(A\otimes \ct,d+\delta+[\omega,-])\simeq H_n(\Der(\ct),\ad(\delta)).\]
\end{proof}

The next theorem is also proved in \cite{KMT}.
\begin{thm}[Kajiura-M.-Terashima]\label{ThKMT}
The set $\pi_0(Q_\bullet(X,\delta))$ has the standard right action of $\QA^{(1)}(\delta)$, and the action is free and transitive. 

\end{thm}
\begin{proof}
There are the identifications obtained in Section \ref{homotopyalg}:
\[\pi_0(Q_\bullet(X,\delta))=\{\text{$C_\infty$-algebra model }f:(H,m)\to A\}/(C_\infty\text{-homotopic}),\]
\[\QA^{(1)}(\delta)=\{\text{$C_\infty$-isom }f:(H,m)\to A;f_1=\id_H:H\to H\}/(C_\infty\text{-homotopic}),\]
where $m$ is the minimal $C_\infty$-algebra structure on $H$ corresponding to $\delta$. So $\QA^{(1)}(\delta)$ acts on the right of $\pi_0(Q_\bullet(X,\delta))$ by composition. Since any $C_\infty$-quasi-isomorphism has a homotopy inverse, the action is free and transitive.
\end{proof}

\subsection{The simplicial bundle of formal homology connections}\label{QE}

Let $X\to E\to B$ be a smooth fiber bundle. In the section, we shall define the simplicial bundle of formal homology connections on fibers.

\begin{defi}
We define the simplicial bundle $\Q_\bullet(E)\to S_\bullet(B)$ over the simplicial set $S_\bullet(B)$ of smooth singular simplices $\Delta^n\to B$ as follows:
\begin{itemize}
\item the fiber over an $n$-simplex $\sigma\in S_n(B)$ is $\Q_n(E)_{\sigma}:=Q_0(\sigma^*E)$, and
\item the face maps and the degeneracy maps are the induced maps $\Q_n(E)_{\sigma}\to \Q_{n-1}(E)_{\pt_i\sigma}$ and $\Q_n(E)_{\sigma}\to \Q_{n+1}(E)_{s_i\sigma}$ by the coface maps and the codegeneracy maps of $\Delta^\bullet$ respectively. 
\end{itemize}
\end{defi}

We can check that $\Q_\bullet(E)\to S_\bullet(B)$ is a bundle of simplicial sets in the sense of May \cite{May}. 
\begin{prop}
The simplicial map $\Q_\bullet(E)\to S_\bullet(B)$ is a simplicial bundle with fiber $Q_\bullet(X)$. 
\end{prop}
\begin{proof}
For an $n$-simplex $\sigma\in S_n(B)$ and a trivialization $\varphi_\sigma:\Delta^n\times X\simeq \sigma^*E$, we obtain the trivialization $\varphi_{\sigma,P}:\Delta^i\times X\simeq \sigma(P)^*E$ for $P\in \Delta[n]_i$ by the diagram
\[\xymatrix{\Delta^n\times X\ar[r]^{\varphi_\sigma}&\sigma^*E\\\Delta^i\times X\ar[u]^{f_P\times\id_X}\ar[r]^{\varphi_{\sigma,P}}&\sigma(P)^*E\ar[u],}\]
regarding $\sigma$ as a simplicial map $\sigma:\Delta[n]\to S_\bullet(B)$. Here the map $f_P:\Delta^i\to \Delta^n$ is the induced map $P:\Delta[i]\to \Delta[n]$.

Then we obtain the simplicial trivialization
\[\hat{\varphi}_\sigma:\sigma^*\Q_\bullet(E)\simeq \Delta[n]\times Q_\bullet(X).\]
by $(P,\alpha)\mapsto (P,\varphi_{\sigma,P}^*\alpha)$, where
\[\sigma^*\Q_i(E)=\{(P,\alpha)\in \Delta[n]_i\times Q_0(\sigma(P)^*E)\}.\]
\end{proof}

We need to define a section of differentials $\delta$ to restrict the simplicial fiber $Q_\bullet(X)$ to $Q_\bullet(X,\delta)$.
\begin{defi}
Fix a Chen differential $\delta\in \Der(\ct)_{-1}$ of $X$. Suppose $\delta$ is $G$-invariant with respect to the action of the homological structure group $G$ on $\Der(\ct)$ (induced by the action on $W$). Then it gives the section $\hat{\delta}$ of the surjective map between sets
\[\CD(E):=\coprod_{b\in B}\CD(E)_b\to B,\]
where $\CD(E)_b:=\{\text{Chen differential of $E_b$}\}$ for $b\in B$. Explicitly, $\hat{\delta}:B\to \CD(E)$ is described by
\[\hat{\delta}(b)=|\varphi_b|^{-1}\circ \delta\circ |\varphi_b|\in \Der(\hat{L}W(E_b))_{-1},\quad W(E_b):=\tilde{H}_\bullet(E_b;\r)[-1]\]
where $|\varphi_b|:\hat{L}W(E_b)\simeq \hat{L}W$ is the isomorphism induced by a trivialization $\varphi_b:E_b\simeq X$.  We call $\hat{\delta}$ a \textbf{section of Chen differentials}. Given this, we can consider the simplicial bundle $\Q_\bullet(E,\hat{\delta})\to S_\bullet(B)$ defined by
\[\Q_n(E,\hat{\delta})_\sigma:=Q_0(\sigma^*E,\hat{\delta}(\sigma))\]
for $\sigma\in S_n(B)$. Here $\hat{\delta}(\sigma)$ is the Chen differential of $\sigma^*E$ defined by $\hat{\delta}(\sigma_0)$ through the isomorphism $\tilde{H}_\bullet(\sigma^*E)\simeq \tilde{H}_\bullet(E_{\sigma_0})$ on homologies. Here $\sigma_0=\pt_1\cdots \pt_n\sigma$ is the image of the base point of $\Delta^n$.

\end{defi}

For example, if $X$ is formal, the differential $\delta$ corresponding to the cohomology ring structure of $X$ is $\Diff(X)$-invariant.

\section{Obstruction theory}\label{OT}

Obstruction theory for simplicial sets is studied in \cite{BFG, DK}. We shall review a part of them and rewrite briefly obstruction theory as in Steenrod \cite{St} for simplicial sets in order to fit our use. 
\begin{itemize}
\item In Section \ref{localsys}, we define the fundamental groupoid, a local system, and the cochain complex with local coefficient. 
\item In Section \ref{obstruction}, we define obstruction classes to extend a section of a simplicial bundle over the $n$-skeleton for $n\geq 1$.
\item In Section \ref{non-abelian}, we suppose the local system $\Pi_0$ of connected components of fibers has a free and transitive action of a local system $\G$ of filtered groups. Under the assumption, we introduce obstruction classes to extend a section of a simplicial bundle over the $0$-skeleton stepwisely using the filtration of $\G$.
\end{itemize}
We shall apply these constructions to the simplicial bundle $Q_\bullet(X)\to \Q_\bullet(E)\to S_\bullet(B)$ in Section \ref{mainsection}.

\subsection{Local system}\label{localsys}

We shall define cohomology with local coefficients briefly. We can see definitions in \cite{BFG, DK}. 
\begin{defi}
Let $\X$ be a Kan complex. We define the \textbf{fundamental groupoid} $\Pi_1(\X)$ of $\X$ such that the set of objects is $\X_0$ and the set of morphisms from $x$ to $y$ is the set of homotopy classes of $\gamma\in \X_1$ satisfying $\pt_0\gamma=x$ and $\pt_1\gamma=y$. A covariant functor $\Pi_1(\X)\to \Ab$ is called a \textbf{local system} on $\X$. Here $\Ab$ is the category of abelian groups.
\end{defi}

Let $\E\to \B$ be a Kan simplicial bundle with $n$-simple fiber $\X$, i.e., $\X$ is a Kan complex and $\pi_1(\X,x)$ acts on $\pi_n(\X,x)$ trivially.

\begin{defi}
We define the local system $\Pi_n(\E/\B)$ on $\B$ as follows: for a vertex $v\in \B_0$,
\[\Pi_n(\E/\B)_v:=\pi_n(v^*\E).\]
Note that we need not to choose a base point of $v^*\E$ because it is $n$-simple. For a path $\gamma\in \B_1$ such that $v_0=\pt_1\gamma$ and $v_1=\pt_0\gamma$, take a trivialization
\[\varphi_\gamma:\Delta[1]\times v_0^*\E\simeq \gamma^*\E\]
such that 
\[\xymatrix{\Delta[1]\times v_0^*\E\ar[r]^-{\varphi_\gamma}& \gamma^*\E\\v_0^*\E\ar[u]^{\delta^1}\ar@{=}[r]&v_0^*\E\ar[u]_{\text{incl.}}.}\]
Here $\delta^i:\Delta[0]\to \Delta[1]$ is the coface maps. Then we have the isomorphism $g_\gamma:v_0^*\E\to v_1^*\E$, which is called the \textbf{holonomy} along $\gamma$, defined by
\[\xymatrix{\Delta[1]\times v_0^*\E\ar[r]^-{\varphi_\gamma}&\gamma^*\E\\v_0^*\E\ar[u]^{\delta^0}\ar[r]_{g_\gamma}&v_1^*\E\ar[u]_{\text{incl.}}.}\]
So we put
\[\Pi_n(\E/\B)(\gamma):=(g_\gamma^{-1})_*:\pi_n(v_1^*\E)\to \pi_n(v_0^*\E).\]
We can prove that it is depend on only the homotopy class of $\gamma$ since $\E\to \B$ is Kan fibration. In fact, for another path $\gamma'$ homotopic to $\gamma$ by a homotopy $\sigma\in \B_2$, there exists a homotopy $h$ satisfying the commutative diagram
\[\xymatrix{\Lambda^2[2]\times v_0^*\E\ar[r]^-{\varphi_\gamma\cup\varphi_{\gamma'}}\ar[d]&\sigma^*\E\ar[d]\\\Delta[2]\times v_0^*\E\ar@{.>}[ur]^h\ar[r]&\Delta[2]}\]
by Theorem 7.8 in \cite{May}. Here $\Lambda^2[2]$ is the $(2,2)$-horn.
\end{defi}

The cochain complex and the cohomology with local coefficients are defined as follows.

\begin{defi}
Let $\X$ be a Kan complex, $\A$ a simplicial subset of $\X$, and $M:\Pi_1(\X)\to \Ab$ a local system on $\X$. We define \textbf{the cochain complex with coefficient $M$} by 
\[C^n(\X,\A;M):=\left\{c:\X_n\to \coprod_{v\in \X_0}M(v);\ c(x)\in M(x_0),\ c|\A=0\right\},\]
where $x_0=\pt_1\cdots \pt_nx$, and its normalized version by
\[N^n(\X,\A;M):=\bigcap_{i=0}^n\Ker (s_i^*:C^n(\X,\A;M)\to C^{n-1}(\X,\A;M)).\]
The differential $\delta:C^n(\X,\A;M)\to C^{n+1}(\X,\A;M)$ is defined by
\[\delta c(x)=M(x_{01})^{-1}c(\pt_0x)-c(\pt_1x)+\cdots +(-1)^{n+1}c(\pt_{n+1}x),\]
where $x_{01}=\pt_2\cdots \pt_nx$. Its cohomology is denoted by $H^n(\X,\A;M)$.

\end{defi}

\subsection{Obstruction cocycles and difference cochains}\label{obstruction}
Let $\A$ be a simplicial subset of $\B$. We call a simplicial map $s$ satisfying the following diagram an \textbf{$n$-partial section} relative to $\A$:
\[\xymatrix{&\E\ar[d]\\\sk_n(\B)\cup\A\ar[ur]^s\ar[r]&\B}\]

Given an $n$-partial section $s:\sk_n(\B)\cup\A\to \E$ relative to $\A$, we shall construct the \textbf{obstruction cocycle} of $s$
\[c(s)\in N^{n+1}(\B,\A;\Pi_n(\E/\B))\]
to extend a partial section $\sk_{n+1}(\B)\cup\A\to \E$ as follows: for an $(n+1)$-simplex $\sigma\in \B_{n+1}$, we get the induced section $s_\sigma$ such that
\[\xymatrix{\sigma^*\E\ar[r]&\E\\\sk_n(\Delta[n+1])\ar@{.>}[u]^{s_\sigma}\ar[r]^-{\sk_n(\sigma)}&\sk_n(\B).\ar[u]^s}\]
So we put
\[c(s)(\sigma):=g_\sigma^{-1}[s_\sigma]\in \pi_n(\sigma_0^*\E),\]
where $g_\sigma:\pi_n(\sigma_0^*\E)\to \pi_n(\sigma^*\E)$ is an isomorphism induced by the inclusion $\sigma_0^*\E\to\sigma^*\E$. 

\begin{prop}
The cochain $c(s)$ is a cocycle.

\end{prop}

\begin{proof}
For an $(n+2)$-simplex $\sigma\in \B_{n+2}$, we have 
\[\xymatrix{(\pt_i\sigma)^*\E\ar[r]&\sigma^*\E\ar[r]&\E\\\sk_n(\Delta[n+1])\ar[r]_{\sk_n(\delta^i)}\ar[u]^{s_{\pt_i\sigma}}&\sk_n(\Delta[n+2])\ar[r]_-{\sk_n(\sigma)}\ar[u]^{s_\sigma}&\sk_n(\B).\ar[u]^s}\]
So the commutative diagrams for $i\neq0$ 
\[\xymatrix{\sigma_0^*\E\ar[dr]\ar[d]&\\(\pt_i\sigma)^*\E\ar[r]&\sigma^*\E}\quad \xymatrix{\sigma_1^*\E\ar[d]&\sigma_0^*\E\ar[d]\ar[l]_{g_{\sigma_{01}}}\\(\pt_0\sigma)^*\E\ar[r]&\sigma^*\E}\]
imply the equations
\[g_{\pt_i\sigma}^{-1}[s_{\pt_i\sigma}]=g_{\sigma}^{-1}(s_\sigma)_*[\sk_n(\delta^i)],\quad g_{\sigma_{01}}^{-1}g_{\pt_0\sigma}^{-1}[s_{\pt_0\sigma}]=g_\sigma^{-1}(s_\sigma)_*(\sigma_{01})_*[\sk_n(\delta^0)].\]
Here note that $[\sk_n(\delta^i)]\in \pi_n(\sk_n(\Delta[n+2]),0)$ and $[\sk_n(\delta^0)]\in \pi_n(\sk_n(\Delta[n+2]),1)$. Thus we obtain
\[(\delta c(s))(\sigma)=g_\sigma^{-1}(s_\sigma)_*\left((\sigma_{01})_*[\sk_n(\delta^0)]+\sum_{i\neq0}(-1)^i[\sk_n(\delta^i)]\right)=0,\]
using the relation $(\sigma_{01})_*[\sk_n(\delta^0)]+\sum_{i\neq0}(-1)^i[\sk_n(\delta^i)]=0$ in $\pi_n(\sk_n(\Delta[n+2]),0)$.
\end{proof}

We shall define the difference cochain for $n$-partial sections $s_0,s_1:\sk_n(\B)\to \E$ and a fiberwise homotopy $h:\sk_{n-1}(\B)\times \Delta[1]\to \E\times\Delta[1]$ between their restriction on $\sk_{n-1}(\B)$. Gluing these maps, we have the map
\[h^{\square}:(\sk_n(\B)\times\sk_0(\Delta[1]))\cup(\sk_{n-1}(\B)\times \Delta[1])\to \E\times\Delta[1].\]
We consider the obstruction cocycle
\[c(h^{\square})\in N^{n+1}(\sk_n(\B)\times\Delta[1],(\sk_n(\B)\times\sk_0(\Delta[1]))\cup(\sk_{n-1}(\B)\times \Delta[1]);\Pi_n^{\square}),\]
where $\Pi_n^{\square}=\Pi_n(\E\times\Delta[1]/\B\times\Delta[1])$.
Note that faces of non-degenerate simplices of $\sk_n(\B)\times \Delta[1]$ are in $(\sk_n(\B)\times\sk_0(\Delta[1]))\cup(\sk_{n-1}(\B)\times \Delta[1])$. Through the Eilenberg-Zilber map
\[\times:N_n(\B)\otimes N_1(\Delta[1])\to N_{n+1}(\sk_n(\B)\times\Delta[1],(\sk_n(\B)\times\sk_0(\Delta[1]))\cup(\sk_{n-1}(\B)\times \Delta[1])),\]
we can define the cochain $d(s_0,h,s_1)\in N^n(\B;\Pi_n(\E/\B))$ by
\[d(s_0,h,s_1)(\sigma):=(-1)^nc(h^{\square})(\sigma\times I)\]
for $\sigma\in \B_n$. Here $I$ is the unique non-degenerate simplex in $\Delta[1]_1$.
\begin{prop}
The cochain $d(s_0,h,s_1)$ satisfies
\[\delta d(s_0,h,s_1)=c(s_1)-c(s_0).\]
\end{prop}
\begin{proof}
It is proved by the equations
\begin{align*}
\delta d(s_0,h,s_1)(\sigma)&=g_{\sigma_{01}}^{-1}d(s_0,h,s_1)(\pt_0\sigma)+\sum_{i\neq0}(-1)^id(s_0,h,s_1)(\pt_i\sigma)\\
&=(-1)^ng_{\sigma_{01}}^{-1}c(h^{\square})(\pt_0\sigma\otimes I)+\sum_{i\neq0}(-1)^{n+i}c(h^{\square})(\pt_i\sigma\otimes I)\\
&=c(h^{\square})(\sigma\otimes \pt I)-\delta c(h^{\square})(\sigma\otimes I)\\
&=c(s_1)-c(s_0).\end{align*}
\end{proof}

The next two propositions hold in the same way as in obstruction theory \cite{St}.
\begin{prop}
An $n$-partial section $s:\sk_n(\B)\to \E$ extends to an $(n+1)$-partial section $\sk_{n+1}(\B)\to \E$ if and only if $c(s)=0$.

\end{prop}
\begin{prop}
For $n$-partial sections $s,s':\sk_n(\B)\to \E$, if obstruction cocycles $c(s)$ and $c(s')$ are cohomologous, there is a homotopy between $s|\sk_{n-1}(\B)$ and $s'|\sk_{n-1}(\B)$.

\end{prop}

Suppose a fiber $\X$ of a Kan fiber bundle $\E\to\B$ is $(n-1)$-connected (and $\pi_1(\X,x)$ is abelian if $n=1$). Then we can get an $n$-partial section $s:\sk_n(\B)\to \E$. If we get another $n$-partial section $s'$, these is a homotopy between $s|\sk_{n-1}(\B)$ and $s'|\sk_{n-1}(\B)$. So we obtain an invariant 
\[\o_n(\E):=[c(s)]\in H^{n+1}(\B;\Pi_n(\E/\B)).\]
It is called the \textbf{obstruction class} of $\E\to\B$.

\subsection{Obstruction for $n=0$}\label{non-abelian}

We consider an extension of a $0$-partial section under the following situation: for a simplicial bundle $\E\to\B$, suppose that the local system $\Pi_0(\E/\B)$ of sets has a free and transitive right action of a local system $\G$ of groups on $\B$. 

At first, we define the non-abelian obstruction class of a $0$-partial section. For that, we remark the definition of the non-abelian cohomology with values in a local system of non-abelian groups. Here ``non-abelian cohomology'' is in the sense of \cite{Fre}.

\begin{defi}
Let $\X$ be a simplicial set and $\G$ a local system of groups on $\X$. Define the \textbf{(non-abelian) cochain complex} of $\X$ with coefficient $\G$ 
\[C^n(\X;\G):=\left\{c:\X_n\to \coprod_{v\in \X_0}\G(v);\ c(x)\in \G(x_0)\right\}\]
for $0\leq n\leq 2$ and the following datum:
\begin{enumerate}
\item the affine action $\varphi$ of $C^0(\X;\G)$ on $C^1(\X;\G)$:
\[(\varphi(f)c)(\gamma)=f(\pt_1\gamma)c(\gamma)(\G(\gamma)^{-1}f(\pt_0\gamma)^{-1})\]
for $f\in C^0(\X;\G)$ and $c\in C^1(\X;\G)$,
\item the action $\psi$ of $C^0(\X;\G)$ on $C^2(\X;\G)$:
\[(\psi(f)c)(\sigma)=\Ad(\G(\pt_2\sigma)^{-1}f(\pt_0\pt_2\sigma))(c(\sigma))\]
for $f\in C^0(\X;\G)$ and $c\in C^2(\X;\G)$,
\item the map $\delta:C^1(\X;\G) \to C^2(\X;\G)$ satisfying $\delta(1)=1$ and $\delta(\varphi(f)c)=\psi(f)c$ for $f\in C^0(\X;\G)$ and $c\in C^1(\X;\G)$:
\[\delta c(\sigma)=(\G(\pt_2\sigma)^{-1}c(\pt_0\sigma))c(\pt_1\sigma)^{-1}c(\pt_2\sigma)\]
for $c\in C^1(\X;\G)$ and $\sigma\in \X_2$.
\end{enumerate}
The we get the \textbf{0-th cohomology group}
\[H^0(\X;\G):=\Ker (C^0(\X;\G)\to \Aut(C^1(\X;\G))\ltimes C^1(\X;\G)\to C^1(\X;\G))\]
and the \textbf{1-st cohomology set}
\[H^1(\X;\G):=\delta^{-1}(1)/C^0(\X;\G).\] 
\end{defi}

Given a $0$-partial section $s:\sk_0(\B)\to \E$, put
\[c(s)(\gamma)=[s(\pt_1\gamma)]^{-1}(\Pi_0(\gamma)^{-1}[s(\pt_0\gamma)])\in \G_{\gamma_0}\]
for $\gamma\in\B_1$, i.e., $c(s)(\gamma)\in\G_{\gamma_0}$ is the unique element satisfying
\[[s(\pt_1\gamma)]c(s)(\gamma)=\Pi_0(\gamma)^{-1}[s(\pt_0\gamma)].\]
By definition, $c(s)\in C^1(\B;\G)$ is a cocycle. For another section $s':\sk_0(\B)\to \E$, if we can get $f\in C^0(\B;\G)$ uniquely such that 
\[s'(x)=s(x)f(x)\]
for $x\in X_0$, then $c(s')=\varphi(f)c(s)$ holds. We denote $f$ by $d(s,s')$ as in Section \ref{obstruction}. Especially the cohomology class \[\o_0(\E):=[c(s)]\in H^1(\B;\G)\] is independent of a choice of a $0$-partial section $s:\sk_0(\B)\to \E$. As usual obstruction theory, $\o_0(\E)=1$ if and only if there is a $1$-partial section $\sk_1(\B)\to\E$. It follows from the following proposition:
\begin{prop}
If $\o_0(\E)=1$, there exists a $0$-partial section $s:\sk_0(\B)\to \E$ such that $c(s)=1$.

\end{prop}
\begin{proof}
If $[c(s)]=1$, there exists $f\in C^0(\B;\G)$ such that $c(s)=\varphi(f)(1)$. So replacing $s$ with $sf^{-1}$, we get the proposition.
\end{proof}

The non-abelian obstruction $\o_0(\E)$ is hard to deal with. So we shall consider to get a certain abelian cocycle using a filtration $\{\F^{(i)}\G\}_{i=1}^\infty$ of $\G$ such that 
\[\G_b=\F^{(1)}\G_b\rhd \F^{(2)}\G_b\rhd\cdots ,\]
\[[\F^{(i)}\G_b,\F^{(j)}\G_b]\subset \F^{(i+j)}\G_b\]
for $b\in \B_0$, and the map $\G(\gamma)$ for $\gamma\in\B_1$ preserves the filtration. Given such a filtration, we have the local system of abelian groups defined by
\[\gr_i(\G):=\F^{(i)}\G/\F^{(i+1)}\G.\]
It is also written by $\gr_i(\G)_b=\gr_i(\G_b)$ for $b\in \B_0$.

If the image of $c(s)$ to $C^1(\B;\G/\F^{(i)}\G)$ is trivial, i.e., $c(s)(\gamma)\in \F^{(i)}\G_{\gamma_0}$ for $\gamma\in \B_1$, we get its image $c_i(s)$ to the (abelian) chain complex $C^1(\B;\gr_i(\G))$. For another partial section $s':\sk_0(\B)\to \E$ satisfying the same condition, we can also get the image $d_i(s,s')$ of $d(s,s')$ to $C^1(\B;\gr_i(\G))$. Then it satisfies the equation
\[c_i(s')-c_i(s)=\delta d_i(s,s').\]
It means $\o^{(i)}(\E):=[c_i(s)]\in H^1(\B;\gr_i(\G))$ is obtained uniquely.
\begin{prop}
If $\o^{(i)}(\E)$ is defined and trivial, there exists a partial section $s:\sk_0(\B)\to \E$ such that $c(s)(\gamma)\in \F^{(i+1)}\G_{\gamma_0}$ for $\gamma\in \B_1$.

\end{prop}
\begin{proof}
Supposing $\o^{(i)}(\E)=[c_i(s)]=1$, we have $1=[c(s)]\in H^1(\B;\G/\F^{(i+1)}\G)$. Then there exists a $0$-partial section $s':\sk_0(\B)\to \E$ such that $c(s')=1\in C^1(\B;\G/\F^{(i+1)}\G)$. This section satisfies the required condition.
\end{proof}

\section{Obstruction of the bundles of formal homology connections}\label{mainsection}

Let $X\to E\to B$ be a smooth fiber bundle with homological structure group $G$. Fix a $G$-invariant Chen differential $\delta$ on $\ct$, where $W=\tilde{H}_\bullet(X;\r)[-1]$.

In this section, we shall apply the construction of the obstruction class in Section \ref{OT} for $\Q_\bullet(E)\to S_\bullet(B)$ and get characteristic classes using the obstruction. 
\begin{itemize}
\item In Section \ref{connected}, we consider the case that $Q_\bullet(X)$ is connected.
\item In Section \ref{Exsp}, we give a non-trivial example of the construction in Section \ref{connected}.
\item In Section \ref{nonconnected}, we consider the case that $Q_\bullet(X)$ is not connected.
\item In Section \ref{Exsur}, we give a non-trivial example of the construction in Section \ref{nonconnected}. We shall prove that the our obstruction class for a surface bundle is equivalent to the 1st Morita-Miller-Mumford class.
\end{itemize}

\subsection{Connected cases}\label{connected}
Suppose $\H_0^{(1)}(\delta)=0$ and $\H_i(\delta)=0$ for $n>i>0$. In addition, suppose, if $n=1$, $\H_1(\delta)\simeq H_0(\hat{L}W\otimes A,d+\delta+[\tau,-])$ is abelian with respect to the Baker-Campbell-Hausdorff product. Then we get the obstruction class of the simplicial bundle $\Q_\bullet(E,\hat{\delta})\to S_\bullet(B)$
\[\o=\o_n(\Q_\bullet(E,\hat{\delta}))\in H^{n+1}(B;\Pi_n),\]
where $\Pi_n=\Pi_n(\Q_\bullet(E,\hat{\delta})/S_\bullet(B))$, and the characteristic maps of a fiber bundle $E\to B$
\[(\Lambda^p\H_n(\delta)^*)^G\to H^{p(n+1)}(B;\r)\]
by $\psi\mapsto \psi(\o,\dots,\o)$ for $p\geq 1$. Strictly speaking, the map is defined as follows. First take a family $\{\varphi_i:E|U_i\simeq U_i\times X\}$ of local trivializations of $E$ over an open covering $\{U_i\}$ of $B$ such that the images of transition functions is included in the structure group. For a simplex $\sigma\in S_{n+1}(B)$, choose an open set $U_i$ containing a vertex $b=\sigma_0$. Then the trivialization $\varphi_i$ induces the isomorphism $g_b:\Pi_n(b)= \H_n(\hat{\delta}(b))\simeq\H_n(\delta)$. If $p=1$, put $\psi(\o)(\sigma):=[\psi(g_bc(\sigma))]$ when $\o$ is represented by a cochain $c$. For a generic integer $p>1$, the value of $\psi$ is defined by the cup product $\psi(\o,\dots,\o):=\sum\psi_{j_1}(\o)\cdots\psi_{j_p}(\o)$ when $\psi$ is described by the sum $\sum\psi_{j_1}\cdots \psi_{j_p}$ for $\psi_j\in \H_n(\delta)^*$. It is independent of choices of open sets and a family of trivializations because of $G$-invariance of $\psi$.

\subsection{Example of a sphere bundle}\label{Exsp}

We consider the sphere bundle $S^2\to E=S^3\times_{S^1}S^2\to S^2$ associated to the Hopf fibration $S^1\to S^3\to S^2$, where $U(1)=S^1$ acts on $S^2=\c\cup\{\infty\}$ by rotations. Here
\[S^3=\{(z_1,z_2)\in \c^2;\ |z_1|^2+|z_2|^2=1\},\]
\[E=S^3\times_{S^1}S^2=S^3\times S^2/((z,w)\sim (\zeta z,\zeta^{-1} w),\ \zeta\in S^1),\]
where $S^1=\{\zeta\in \c;\ |\zeta|=1\}$. The image of $(z,w)\in S^3\times S^2$ in $E$ is denoted by $[z,w]$. We use the identification $S^3/S^1\simeq S^2=\c\cup \{\infty\}$ defined by $[z_1,z_2]\mapsto z_1/z_2$.

Since the action of $S^1$ on $S^2$ has two fixed points $0$ and $\infty$, this fiber bundle has a section $S^2=S^3/S^1\to S^3\times_{S^1}S^2$ defined by $[b]\mapsto [b,\infty]$. We fix the section. 

Denote the volume form on the fiber $S^2=\c \cup\{\infty\}$ by
\[v=\frac{\sqrt{-1}}{2\pi}\frac{dwd\bar{w}}{(1+|w|^2)^2},\]
where $w$ is the complex coordinate of the Riemann sphere $S^2$, and the desuspension of the fundamental class by $x=\sus^{-1}[v]\in W=H_2(S^2)[-1]$. Then a DGL model $(LW,\delta)$ of $S^2$ is given by
\[LW=L(x)\ (|x|=1),\quad \delta x=0\]
and its Lie algebra of derivations 
\[\Der(LW)=\left\langle x\frac{\pt}{\pt x},[x,x]\frac{\pt}{\pt x}\right\rangle.\]
Here $\braket{v_1,\dots,v_n}$ is the vector space generated by $v_1,\dots,v_n$, and $L(v_1,\dots,v_n)$ is the Lie algebra generated by $v_1,\dots,v_n$. We also get 
\[\H_1(\delta)=\Der(LW)_1=\left\langle [x,x]\frac{\pt}{\pt x}\right\rangle.\]

For simplicity, we restrict the bundle $\Q_\bullet(E)\to S_\bullet(S^2)$ to the Kan complex defined by
\[K_n=\{(\Delta^n,\sk_1\Delta^n)\to (S^2,\infty)\}\subset S_n(S^2).\]
If $n\leq 1$, $K_n$ is described by
\[K_0=\{p_\infty\},\quad K_1=\{\gamma_\infty\},\]
where $p_\infty:\Delta^0\to S^2$ and $\gamma_\infty:\Delta^1\to S^2$ are constant maps to the point $\infty$. We put $\Q_\bullet:=\Q_\bullet(E)|_{K_\bullet}$.

Put $D^2=\{z\in\c;|z|\leq1\}\subset \c$. We use the map $\rho:D^2\to S^2$ defined by 
\[\rho(z)=\begin{cases}2z/(1-|z|^2)&(|z|<1)\\\infty&(|z|=1),\end{cases}\]
the trivialization $\varphi_\rho:D^2\times S^2\to \rho^*E$ defined by
\[\varphi_\rho(z,w)=\left(z,\left[\left(\frac{2z}{1+|z|^2},\frac{1-|z|^2}{1+|z|^2}\right),w\right]\right).\]
Choose an orientation-preserving diffeomorphism $h:\Delta^2/(\delta^1\Delta^1\cup\delta^2\Delta^1)\simeq D^2$ such that 
\[[0,1]= \Delta^1\overset{\delta^0}\to \Delta^2/(\delta^1\Delta^1\cup\delta^2\Delta^1)\overset{h}\to D^2\]
is given by $t\mapsto e^{2\pi\sqrt{-1}t}$. Here the standard 1-simplex $\Delta^1$ is identified with the interval $[0,1]$ by the isomorphism $[0,1]\simeq \Delta^1:t\mapsto (t,1-t)$. By composing the projection $\Delta^2\to \Delta^2/(\delta^1\Delta^1\cup\delta^2\Delta^1)$, we get $\bar{h}:\Delta^2\to D^2$ and the 2-simplex $\sigma=\rho\bar{h}:\Delta^2\to S^2$ in $K_\bullet$. Then we have the trivialization $\varphi_\sigma:\Delta^2\times S^2\simeq \sigma^*E$ and its restriction $\varphi_{\sigma,\pt_0}:\Delta^1\times S^2\simeq (\pt_0\sigma)^*E=\gamma_\infty^*E$ satisfying
\[\xymatrix{\Delta^1\times S^2\ar[d]^{\varphi_{\sigma,\pt_0}}\ar[r]^{\delta^0\times \id}&\Delta^2\times S^2\ar[d]^{\varphi_\sigma}\ar[r]^{\bar{h}\times \id}&D^2\times S^2\ar[d]^{\varphi_\rho}&\\\gamma_\infty^*E\ar[r]\ar[d]&\sigma^*E\ar[r]\ar[d]&\rho^*E\ar[r]\ar[d]&E\ar[d]\\\Delta^1\ar[r]^{\delta^0}&\Delta^2\ar[r]^{\bar{h}}&D^2\ar[r]^\rho&S^2.}\]
On the other hand, we have the trivialization $\varphi_{\gamma_\infty}:\Delta^1\times S^2\simeq  \Delta^1\times E_\infty=\gamma_\infty^*E $ defined by $(t,w)\mapsto (t,[(1,0),w])$. Then the transition function $g=\varphi_{\gamma_\infty}^{-1}\circ \varphi_{\sigma,\pt_0}:\Delta^1\times S^2\simeq \Delta^1\times S^2$ is described by 
\[g(t,w)=\varphi_{\gamma_\infty}^{-1}(t,[(e^{2\pi\sqrt{-1}t},0),w])=(t,e^{-2\pi\sqrt{-1}t}w).\]
We can also define the trivialization $\varphi_{p_\infty}:S^2\simeq p_\infty^*E=E_\infty$ on $\infty\in S^2$
\[\varphi_{p_\infty}(w)=[(1,0),w].\]

The partial section $s:\sk_1K\to \Q$ is defined as follows: 
\[s(p_\infty):=v_0 x\in \Q_0(E)_{p_\infty},\quad s(\gamma_\infty):=v_1x\in \Q_1(E)_{\gamma_\infty},\]
where $v_0:=(\varphi_{p_\infty}^{-1})^*v\in A^2(E_\infty)$ and $v_1:=(\varphi_{\gamma_\infty}^{-1})^*v\in A^2(\gamma_\infty^*E)$. Remark that $v\in A^2(S^2)\subset A^2(\Delta^1\times S^2)$. According to Section \ref{obstruction}, we get $s_\sigma:\sk_1(\Delta[2])\to\sigma^*\Q_\bullet(E) $. Its homotopy class is equal to $[s_\sigma]=[v_1x]\in \pi_1(\sigma^*\Q_\bullet(E),v_0x)$. So we have
\[c(s)(\sigma)=g^*[s_\sigma]=g^*[v_1x]=[g^*(v_1)x]\in\pi_1(Q_\bullet(S^2),vx)\]
under the identification $\varphi_{p_\infty}^*:\pi_1(Q_\bullet(E_\infty),v_0x)\simeq \pi_1(Q_\bullet(S^2),vx)$. By direct calculation, we get
\[g^*(v)=v+\xi dt,\]
where 
\[\xi=-\frac{\bar{w}dw+wd\bar{w}}{(1+|w|^2)^2}=df,\quad f(w)=\frac12 \frac{1}{1+|w|^2}.\]
Then the 2-form
\[\Xi=t_1\xi dt_2-t_2\xi dt_1+2fdt_1dt_2,\]
satisfies the equation
\[(v+\Xi)^2=2v\Xi=4 fv dt_1dt_2=-4fvdt_0dt_2=-4d(fv(t_0dt_2-t_2dt_0)).\]
So we obtain the formal homology connection $\alpha=(v+\Xi)x-4fv(t_0dt_2-t_2dt_0)[x,x]\in Q_2(S^2)$ satisfying
\[\pt_0\alpha=(v+\xi dt_0)x,\quad \pt_1\alpha=vx+4fvdt_0[x,x],\quad \pt_2\alpha=vx.\]
Therefore the equation
\[[g^*(v_1)x]=[(v+\xi dt_0)x]=[vx+4fvdt_0[x,x]]\in \pi_1(Q_\bullet(S^2),vx)\]
holds. Furthermore
\[\int_{S^2}4fv=\int_{S^2}\frac{\sqrt{-1}}{\pi}\frac{dwd\bar{w}}{(1+|w|^2)^3}=\frac{1}{\pi}\int_0^\infty \frac{2rdr}{(1+r^2)^3}\int_0^{2\pi}d\theta=2\int_0^\infty\frac{dx}{(1+x)^3}=1\]
means that the de Rham cohomology class $[4fv]\in H^2(S^2)$ is non-trivial. According to Theorem 4.10 of \cite{B1}, we have $c(s)(\sigma)\neq0$ and \[\o=[c(s)]\neq0\in H^2(K;\H_1(\delta)).\]

Finally evaluating the class with the dual basis $\nu$ of $[x,x]\pt/\pt x\in \Der(LW)_1$, we get the non-trivial characteristic class
\[\nu(\o)\in H^2(K)=H^2(S^2),\]
which is the Euler class of the sphere bundle $E\to S^2$ (given in \cite{Mil}).

\subsection{Non-connected cases}\label{nonconnected}
If $\H_0^{(1)}(\delta)\neq0$, we can apply the construction in Section \ref{non-abelian}. Put $\Pi_0=\Pi_0(\Q(E,\hat{\delta})/S_\bullet(B))$. From Theorem \ref{ThKMT}, the group $\QA^{(1)}(\hat{\delta}(b))$ acts on $\Pi_0(b)$ freely and transitively. 

The local system $\QA^+(E)$ of groups is defined by
\[\QA^+(E)_b:= \QA^{(1)}(\hat{\delta}(b)),\quad \gamma_*(f):=(g_\gamma^{-1})^*\circ f\circ (g_\gamma)^*\]
for $b\in B$, $\gamma\in S_1(B)$ and $f\in \QA^+(E)_{\gamma(0)}$, where $g_\gamma:E_{\gamma(0)}\to E_{\gamma(1)}$ is the holonomy along $\gamma$. Then we get the non-abelian obstruction class \[\o_0=\o_0(\Q_\bullet(E))\in H^1(B;\QA^+(E))\] in Section \ref{non-abelian}.

Furthermore we have the filtration $\{\QA^{(i)}(E)\}_{i=1}^\infty$ of $\QA^+(E)$ defined in Section \ref{der}. By the observations in Section \ref{der}, there exists the identification as local system of vector spaces
\[\gr_i(\QA^+(E))\simeq \gr_i(\H_0^+(E)),\]
where the local system $\H_0^+(E)$ of Lie algebras is defined in the same way as $\QA^+(E)$. Here note that $\gr(\H_0^+(E))$ is defined similarly to $\gr(\QA^+(E))$ using its filtration.

Suppose we get the obstruction class $\o^{(i)}\neq 0\in H^1(B;\gr_i(\H_0^+(E)))$. In the same way as in Section \ref{connected}, the characteristic map 
\[(\Lambda^\bullet\gr_i(\H_0(\delta))^*)^G\to H^\bullet(B;\r)\]
is obtained by $\psi\mapsto \psi(\o^{(i)},\dots,\o^{(i)})$.

Especially, if $X$ is formal and $\delta$ corresponds to the product of the cohomology $H$ of $X$,
we obtain the characteristic map 
\[(\Lambda^\bullet\H_0^i(\delta)^*)^G\to H^\bullet(B;\r).\]

When $X$ is an oriented closed manifold, we shall clarify a relation between the characteristic map constructed in \cite{KMT} and the construction above. According to Chen's theorem, given a metric of the fiber bundle $E\to B$, we have the map $s:B\to \Q_0(E)$ as follows: for $b\in B$, the metric on $E_b$ gives the formal homology connection $s(b)$ on $E_b$ compatible with the Hodge decomposition of $E_b$.

Composing the natural projection $\Q_0(E)\to \CD(E)$ with $s$, we get the section $\hat{\delta}:B\to \CD(E)$.

\begin{thm}\label{rel:KMT}
Let $X$ be an oriented closed manifold and $E\to B$ be a smooth bundle with section and metric. Suppose the metric gives a section $\hat{\delta}$ of Chen differentials corresponding to a $G$-invariant Chen differential $\delta$ of $X$. Then we have the commutative diagram of chain complexes
\[\xymatrix{C_{CE}^\bullet(\H_0^{(1)}(\delta))^G\ar[r]^-{\Phi}\ar[d]&A^\bullet(B)\ar[d]^\int\\(\Lambda^\bullet\gr_1(\H_0(\delta))^*)^G\ar[r]^-{\Phi_1}&C^\bullet(B;\r),}\]
where the first row map $\Phi$ is the characteristic map in \cite{KMT}, the second row $\Phi_1$ is the characteristic map defined by\[\Phi_1(\zeta):=\sum\zeta_{i_1}(c_1(s))\cdots\zeta_{i_p}(c_1(s)),\]
\[\zeta_i(c_1(s))(\gamma):=\zeta_i(c_1(s)(\gamma))\quad (\gamma\in S_1(B))\]
for $\zeta=\sum\zeta_{i_1}\cdots \zeta_{i_p}\in (\Lambda^p\gr_1(\H_0(\delta))^*)^G$, the first column is the natural projection and the second column $\int$ is the de Rham map.

\end{thm}
\begin{proof}
Take a base point $*$ of $B$ and put the universal covering of $B$
\[\tilde{B}=\{\gamma:[0,1]\to B;\gamma(0)=*\}/(\text{homotopy preserving boundary}).\]
We identify the fiber $E_*$ on $*$ with the typical fiber $X$. 

The smooth map $\mu:\tilde{B}\to Q(X,\delta)$ from the universal cover $\tilde{B}$ of $B$ to the moduli space $Q(X,\delta):=\pi_0(Q_\bullet(X,\delta))$ of $C_\infty$-algebra models of $X$ is defined by 
\[\mu([\gamma])=g_{\gamma}^{-1}\cdot [s(\gamma(1))].\]
Here $g_\gamma:E_*\to E_{\gamma(1)}$ is the holonomy along $\gamma$. The right-invariant Maurer-Cartan form 
\[\eta\in A^1(Q(X,\delta);\H_0^{(1)}(\delta)).\] 
is defined by the right-action of $\QA^+(\delta)$ on $Q(X,\delta)$. So we get the flat connection
\[\eta_\mu:=\mu^*\eta\in A^1(\tilde{B};\H_0^{(1)}(\delta)).\]
On the other hand, we can regard $s$ as the $0$-partial section $s:\sk_0(S_\bullet(B))\to \Q_\bullet(E)$. Its non-abelian obstruction cocycle is described by
\[c(s)(\gamma)=[s(\gamma(0))]^{-1}g_\gamma^{-1}[s(\gamma(1))]=g_l(\mu([l])^{-1}\mu([\gamma l])),\]
where $l$ is a path from $*$ to $\gamma(0)$ and a path $\tilde{\gamma}:[0,1]\to\tilde{B}$ is the lift of $\gamma$ such that $\tilde{\gamma}(0)=[l]$. The map $\Psi:Q(X,\delta)\to Q(X,\delta)$ defined by $\Psi(\alpha)=\mu([l])^{-1}\alpha$ satisfies the differential equation $d\Psi=\Psi \eta$. Thus, solving the equation over the path $\mu\tilde{\gamma}$, we have 
\[\mu([l])^{-1}\mu([\gamma l])=\Psi(\mu\tilde{\gamma}(1))=\sum\int_{\mu\tilde{\gamma}}\eta\cdots \eta.\]
Therefore we get the description using iterated integrals
\[c(s)(\gamma)=g_l\cdot\sum\int_{\mu\tilde{\gamma}}\eta\cdots \eta=g_l\cdot\sum\int_{\tilde{\gamma}}\eta_\mu\cdots \eta_\mu.\]
Its projection to $\gr_1(\H_0(\delta))$ is equal to $c_1(s)(\gamma)=g_l\int_{\tilde{\gamma}}\eta_\mu$ and 
\[\int \Phi(\xi)=\int \xi(\eta_\mu,\dots,\eta_\mu)=\bar{\xi}\left(\int\eta_\mu,\dots,\int\eta_\mu\right)=\Phi_1(\bar{\xi})\in C^p(\tilde{B})\]
for $\xi\in C_{CE}^p(\H_0^{(1)}(\delta))^G$, where $\bar{\xi}$ is the projection of $\xi$. Since the element is $\pi_1(B,*)$-invariant, we can regard it as element in $C^p(B)$.\end{proof}

Furthermore if $c_1(s)=\cdots =c_{i-1}(s)=0$, we get the (cocycle-level) characteristic map $\Phi_i:(\Lambda^\bullet\gr_i(\H_0(\delta))^*)^G\to C^\bullet(B;\r)$ defined using $c_i(s)$ instead of $c_1(s)$. Since $\eta_\mu\in A^1(B;\H_0^{(i)}(\delta))$, the commutative diagram similar to Theorem \ref{rel:KMT} exists. So the construction above using obstructions is the ``leading term" of the characteristic map obtained in \cite{KMT}
\[\Phi:C_{CE}^\bullet(\H_0^{(1)}(\delta))^G\to A^\bullet(B).\]

\subsection{Example of surface bundles}\label{Exsur}

We consider the case of $X=\Sigma_g$, which is the closed oriented surface with genus $g\geq2$. This is a K\"ahler manifold, so 
\[\delta=\omega \frac{\pt}{\pt v}\]
is a Chen differential. Here $v\in W_1$ is the fundamental form of $\Sigma_g$ and $\omega\in [W_0,W_0]$ is the intersection form, i.e., $\omega=\sum_{i=1}^g [x^i,y^i]$ for a symplectic basis $\{x^i,y^i\}$ of $W_0$ with respect to the intersection form of $\Sigma_g$.

\subsubsection{The first obstruction for surface bundles}\label{sur}
For a oriented surface bundle (with section), its homological structure group is in the symplectic group $\Sp(W_0)$ of $W_0$. 

\begin{prop}We have the identification as $\Sp(W_0)$-vector space
\[\H_0^1(\delta)\simeq \Lambda^3W_0.\]
\end{prop}

\begin{proof}
An element $D\in \Der^1(LW)_0$ is described by the form
\[D=D_0+[v,z]\frac{\pt}{\pt v}\]
for $D_0\in \Der^1(LW_0)$ and $z\in W_0$. Then we can calculate the image by $\ad(\delta)$:
\[[\delta,D]=-D_0(\omega)+[\omega,z]\frac{\pt}{\pt v}.\]
So, $D$ is in the kernel if and only if $D_0(\omega)\in (\omega)$, where $(\omega)$ is the Lie ideal in $LW_0$ generated by $\omega$. This condition is equivalent to the condition: $D_0$ induces a derivation on $LW_0/(\omega)$

On the other hand, an element $P\in \Der^0(\ct)_1$ is described by
\[P=\sum b_iv\frac{\pt}{\pt x_i}\]
for $b_i\in \r$, where $\{x_i\}_{i=1}^{2g}$ is a basis of $W_0$. Its image of $\ad(\delta)$ is
\[[\delta,P]=\sum b_i\omega\frac{\pt}{\pt x_i}-P(\omega)\frac{\pt}{\pt v}.\]

Since we can prove $[v,W_0]=\{P(\omega);P\in \Der^0(\ct)_1\}$ by direct calculus, for any $D\in \Der^1(\ct)_0$, there exists $P\in \Der^0(\ct)_1$ such that 
\[D_P:=D+[\delta,P]\in \Der^1(LW_0).\]
Furthermore for another $P'\in \Der^0(\ct)_1$ such that $D_{P'}=D+[\delta,P']\in \Der^1(LW_0)$, their difference $[\delta,P-P']$ is in $\Hom(W_0,\r\omega)\subset \Der^1(LW_0)$. So if $D$ is in the kernel, $D_P$ and $D_{P'}$ induce the same derivation on $LW_0/(\omega)$. Therefore we get the isomorphism
\[\H_0^1(\delta)\simeq \Der^1(LW_0/(\omega)).\]
According to \cite{Mor2}, we have the isomorphism $\Der^1(LW_0/(\omega))\simeq \Lambda^3W_0$.
\end{proof}

By the proposition above, for a oriented surface bundle $E\to B$ with section, we get the obstruction class
\[\o^{(1)}=\o^{(1)}(\Q(E,\hat{\delta}))\in H^1(B;\Lambda^3W_0(E)).\]
Here $\Lambda^3W_0(E)$ is the local system of vector spaces such that the space on $b\in B$ is
\[\Lambda^3W_0(E)_b=\Lambda^3\tilde{H}_1(E_b;\r)[-1].\]
This local system is defined in the same way as $\QA^+(E)$ and $\H_0(E)$. Then we also get the characteristic map
\[(\Lambda^\bullet(\Lambda^3W_0)^*)^{\Sp(W_0)}\to H^\bullet (B;\r).\]

\subsubsection{Twisted Morita-Miller-Mumford class}
Let $\M_{g,*}$ be the mapping class group of $\Sigma_g$ fixing a base point. The action of the diffeomorphism group on $Q_0(\Sigma_g,\delta)$ induces the action of $\M_{g,*}$ on $Q(\Sigma_g,\delta)=\pi_0(Q_\bullet(\Sigma_g,\delta))$.

We shall show that the obstruction $\o^{(1)}$ of a surface bundle can be regarded as the 1st twisted Morita-Miller-Mumford class
\[m_{0,3}=[\tau_1]\in H^1(\M_{g,*};\Lambda^3W_0).\]
Here the cross homomorphism $\tau_1:\M_{g,*}\to \Lambda^3W_0$ is the 1st Johnson map. This map is defined as follows: for $f\in Q(\Sigma_g,\delta)$, the total Johnson map $\tau^f:\M_{g,*}\to \QA^{(1)}(\delta)$ is given by the equation
\[f\circ \tau^f(\varphi)=\varphi\cdot f\]
for $\varphi\in \M_{g,*}$. By composing the quotient map \[\QA^{(1)}(\delta)\to \gr_1(\QA(\delta))=\H_0^1(\delta)=\Lambda^3W_0,\] we get $\tau_1:\M_{g,*}\to \Lambda^3W_0$. The map is independent of a choice of $f$. 

For any $\Sigma_g$-bundle $E\to B$, denote the 1st twisted Morita-Miller-Mumford class of $E\to B$ by 
\[m_{0,3}(E):=\rho^*m_{0,3}\in H^1(\pi_1(B,b);\Lambda^3W_0)\simeq H^1(B;\Lambda^3 W_0(E)),\]
where $\rho:\pi_1(B,b)\to \M_{g,*}$ is the holonomy representation of the surface bundle $E\to B$. Note that $\pi_1(B,b)$ acts on $\Lambda^3W_0$ through $\rho$.
\begin{thm}
The obstruction class of a surface bundle $E\to B$ is equal to the minus of the 1st twisted Morita-Miller-Mumford class of $E\to B$:
\[\o^{(1)}(E)=-m_{0,3}(E)\in H^1(B;\Lambda^3W_0(E)).\]
\end{thm}
\begin{proof}
The following discussion is also used in \cite{Kaw2} and Section 4 of \cite{MT}.

Fix a metric of a $\Sigma_g$-bundle $E\to B$. This metric $\mu$ of $E\to B$ gives a section $s$ of $\Pi_0$. This defines the obstruction class
\[\o(E)=[c(s)]\in H^1(B;\QA^{(1)}(E)).\]
According to the proof of Theorem \ref{rel:KMT}, we have
\[c(s)(\gamma)=\sum\int_{\tilde{\gamma}}\eta_\mu\cdots \eta_\mu=\tau^{\mu(b)}(\rho(\gamma))^{-1}\in \QA^{(1)}(E)_b.\]
for a loop $\gamma$ with a base point $b\in B$. Here $\tilde{\gamma},\mu,\eta_\mu$ are as in the proof of Theorem \ref{rel:KMT}. The equation in $\gr_1(\QA(E))_b=\gr_1(\QA(\delta))\simeq \Lambda^3W_0$ implies
\[c_1(s)(\gamma)=-\tau_1^{\mu(b)}(\rho(\gamma)).\]
So we get $\o^{(1)}(E)=-\rho^*m_{0,3}\in H^1(\pi_1(B,b);\Lambda^3W_0)$. 
\end{proof}

So the obtained characteristic map \[\Lambda^\bullet(\Lambda^3W_0^*)^{\Sp(W_0)}\to H^\bullet(B;\r)\] gives Morita-Miller-Mumford classes of $E\to B$ by the result of \cite{KM}.

\end{document}